\documentclass{amsart}
\usepackage{mathrsfs,amssymb,mathrsfs, multirow,setspace}
\usepackage{amsmath}
\usepackage[a4paper, total={7in, 9in}]{geometry}
\usepackage{enumerate}
\theoremstyle{plain}
\usepackage{graphicx}
\usepackage{mathtools}
\usepackage[utf8]{inputenc}

\newtheorem{theorem}{Theorem}[section]
\newtheorem{lemma}[theorem]{Lemma}

\newtheorem{corollary}[theorem]{Corollary}

\theoremstyle{definition}

\newtheorem{example}[theorem]{Example}

\theoremstyle{remark}
\newtheorem{remark}[theorem]{Remark}

\input xy
\xyoption{all}

\renewcommand{\Bbb}{\mathbb} 

\begin{document}
	\title[Laplacian spectrum of weakly zero-divisor graph of the ring $\mathbb{Z}_{n}$]{Laplacian spectrum of weakly zero-divisor graph of the ring $\mathbb{Z}_{n}$}
  \author[Mohd Shariq, Praveen Mathil, Jitender Kumar]{ Mohd Shariq, Praveen Mathil, Jitender Kumar$^{*}$}
   \address{Department of Mathematics, Birla Institute of Technology and Science Pilani, Pilani-333031, India}
 \email{shariqamu90@gmail.com, maithilpraveen@gmail.com,  jitenderarora09@gmail.com}

\begin{abstract}
  Let $R$ be a commutative ring with unity. The weakly zero-divisor graph $W\Gamma(R)$ of the ring $R$ is the simple undirected  graph whose vertices are nonzero zero-divisors of $R$ and two vertices $x$, $y$ are adjacent if and only if there exists $r\in {\rm ann}(x)$ and $s \in {\rm ann}(y)$ such that $rs =0$. The zero-divisor graph of a ring is a spanning subgraph of the weakly zero-divisor graph. It is known that the zero-divisor graph of the ring $\mathbb{Z}_{{p^t}}$, where $p$ is a prime, is the Laplacian integral. In this paper, we obtain the Laplacian spectrum of the weakly zero-divisor graph $W\Gamma(\mathbb{Z}_{n})$ of the ring $\mathbb{Z}_{n}$ and show that  $W\Gamma(\mathbb{Z}_{n})$ is Laplacian integral for arbitrary $n$. 
\end{abstract}
\keywords{ Weakly zero-divisor graph, ring of integers modulo $n$, Laplacian spectrum \\ *  Corresponding author}
\maketitle

\section{Introduction}

Exploring algebraic structures through graph theory has become a captivating research field over the past three decades. This discipline has not only delivered intriguing and exciting outcomes but has also introduced an entirely unexplored domain.
Researchers have extensively studied graphs associated with algebraic structures such as groups and rings, viz Cayley graphs, power graphs, zero-divisor graphs, cozero-divisor graphs,  weakly zero-divisor graphs, and co-maximal graphs, etc. Such  study provides interconnections between algebra and graph theory. 
The zero-divisor graph $\Gamma(R)$ of a commutative ring $R$ 
is the simple undirected graph with vertices non-zero zero-divisors of $R$ and two distinct vertices $x, y$
are adjacent if $xy = 0$. The zero-divisor graphs of rings have been studied extensively by various authors, see
\cite{MR2016655, MR2043378,MR2762487,MR1700509,MR2354873}  and references therein.
The Laplacian spectrum of the zero-divisor graph of the ring $\mathbb Z_n$ was examined by Chattopadhyay et al. \cite{MR4011590}. They demonstrated that the zero divisor graph $\Gamma(\mathbb Z_{p^t})$ of the ring $\mathbb Z_{p^t}$ is Laplacian integral for every prime $p$ and positive integer $t \geq 2$. Various other spectrums of zero divisor graphs have been studied in   \cite{magi2020spectrum, MR4477153, pirzada2021signless, pirzada2020distance, pirzada2021normalized, rather2021laplacian}.

The weakly zero-divisor graph was introduced by  M. J. Nikmehr et al.  \cite{nikmehr2021weakly}. The \emph{weakly zero-divisor graph} $W\Gamma(R)$ of the ring $R$ is the simple undirected graph whose vertex set is the set of all nonzero zero-divisors of $R$  and two vertices $x, y$ are adjacent if and only if there exists $r\in {\rm ann}(x)$ and $s \in {\rm ann}(y)$ such that $rs =0$.  It is easy to observe that the zero-divisor graph of a ring is a spanning subgraph of the weakly zero-divisor graph. Along with the relations between the zero-divisor graph and the weakly zero-divisor graph of rings, the authors of \cite{nikmehr2021weakly} studied the basic properties, including completeness, girth, clique number and vertex chromatic number etc., of the weakly zero-divisor graph. Moreover, all the rings whose weakly zero-divisor graph is a star graph, unicyclic, tree, and split graph, respectively, have been classified in \cite{rehman2022planarity}. Further, they determined all the rings $R$ for which $W\Gamma(R)$ is planar, toroidal, bitoroidal, and of crosscap almost two, respectively. 
  
Motivated by the study of various spectrums of zero-divisor graphs, in this paper, we study the Laplacian spectrum of the weakly zero-divisor graph associated with the ring $\mathbb{Z}_n$. The paper is arranged as follows: In Section 2, we recall the necessary definitions and results. In Section 3, we discuss the structure of $W\Gamma(\mathbb{Z}_n)$. Section 4 obtains the Laplacian spectrum of the weakly zero-divisor graph of the ring $\mathbb{Z}_n$ for arbitrary $n$.
\section{preliminaries}
This section recalls necessary definitions and results. Also, we fix our notations which are used throughout the paper. For unexplained terms of graph theory, we refer the reader to \cite{westgraph}.
A graph $ \Gamma = (V, E)$, where $V = V(\Gamma)$ and $E = E(\Gamma)$ are the set of vertices and the set of edges of $\Gamma$, respectively. 
 Two distinct vertices $x, y \in \Gamma$ are $\mathit{adjacent}$, denoted by $x \sim y$, if there is an edge between $x$ and $y$. Otherwise, we denote it by $x \nsim y$.   
The set $N_{\Gamma}(x)$ of all the vertices adjacent to $x$ in $\Gamma$ is said to be the \emph{neighborhood} of $x$. Additionally, we denote $N[x] = N(x) \cup \{x\}$. 
The graph used in this paper is a simple graph, i.e. undirected graph with no loops or repeated edges. 
 A \emph{subgraph} $\Gamma'$ of a graph $\Gamma$ is a graph such that $V(\Gamma') \subseteq V(\Gamma)$ and $E(\Gamma') \subseteq E(\Gamma)$. If $U \subseteq V(\Gamma)$, then the subgraph of $\Gamma$ induced by $U$, denoted by $\Gamma(U)$, is the graph with vertex set $U$ and two vertices of $\Gamma(U)$ are adjacent if and only if they are adjacent in $\Gamma$. The \emph{complement} $\overline{\Gamma}$ of the graph $\Gamma$ is a graph with the same vertex set as $\Gamma$ and distinct vertices $x, y$ are adjacent in $\overline{\Gamma}$ if they are not adjacent in $\Gamma$. A graph $\Gamma$ is said to be $complete$ if every two distinct vertices are adjacent. The complete graph on $n$ vertices is denoted by $K_n$. A path in a graph is a sequence of distinct vertices with the property that each vertex in the sequence is adjacent to the next vertex of it. The graph $\Gamma$ is said to be \emph{connected} if there is a path between every pair of vertex. The distance $d(x,y)$ between two vertices $x$ and $y$ in a graph is defined as the minimum number of edges in a shortest path connecting them. The \emph{degree} $\text{deg}(v)$ of a vertex $v \in \Gamma$, is the number of edges adjacent to $v$. The \emph{union} $\Gamma_1 \cup \Gamma_2$ is the graph with $V(\Gamma_1 \cup \Gamma_2) = V(\Gamma_1) \cup V(\Gamma_2)$ and $E(\Gamma_1 \cup \Gamma_2) = E(\Gamma_1) \cup E(\Gamma_2)$. The \emph{join} $\Gamma_1 \vee \Gamma_2$ of $\Gamma_1$ and $\Gamma_2$ is the graph obtained from the union of $\Gamma_1$ and $\Gamma_2$ by adding new edges from each vertex of $\Gamma_1$ to every vertex of $\Gamma_2$. Let $\Gamma$ be a graph on $k$ vertices and $V(\Gamma) = \{u_1, u_2, \cdots, u_k\}$. Suppose that $\Gamma_1, \Gamma_2, \cdots, \Gamma_k$ are $k$ pairwise disjoint graphs. Then the \emph{generalised join graph} $\Gamma[\Gamma_1, \Gamma_2, \cdots, \Gamma_k]$ of $\Gamma_1, \Gamma_2, \cdots, \Gamma_k$ is the graph formed by replacing each vertex $u_i$ of $\Gamma$ by $\Gamma_i$ and then joining each vertex of $\Gamma_i$ to every vertex of $\Gamma_j$ whenever $u_i \sim u_j$ in $\Gamma$. 
 
For a finite simple graph $\Gamma$ with vertex set $V(\Gamma) = \{u_1, u_2, \ldots, u_k\}$, the \emph{adjacency matrix} $A(\Gamma)$ is defined as the $k\times k$ matrix whose $(i, j)th$ entry is $1$ if $u_i \sim u_j$, and $0$ otherwise. We denote the diagonal matrix by $D(\Gamma) = {\rm diag}(d_1, d_2, \ldots, d_k)$, where $d_i$ is the degree of the vertex $u_i$ of $\Gamma$. The \emph{Laplacian matrix} $\mathcal{L}(\Gamma)$ of the graph $\Gamma$ is the matrix $D(\Gamma) - A(\Gamma)$. The matrix $\mathcal{L}(\Gamma)$ is a symmetric and positive semidefinite so that its eigenvalues are real and non-negative. Furthermore, the sum of each row (column) of $\mathcal{L}(\Gamma)$ is zero.
The \emph{characteristic polynomial} of $\mathcal{L}(\Gamma)$ is denoted by $\Phi(\mathcal{L}(\Gamma), x)$. 
  The eigenvalues of $\mathcal{L}(\Gamma)$ are called the \emph{Laplacian eigenvalues} of $\Gamma$. If all the eigenvalues of  $\mathcal{L}(\Gamma)$ are integers, then the graph $\Gamma$ is said to be $\emph{Laplacian integral}$. The second smallest Laplacian eigenvalue of $\mathcal{L}(\Gamma)$, denoted by $\mu(\Gamma)$, is called the \emph{algebraic connectivity} of $\Gamma$. The largest Laplacian eigenvalue $\lambda(\Gamma)$ of $\mathcal{L}(\Gamma)$ is called the \emph{Laplacian spectral radius} of $\Gamma$. Let $\lambda_{1}(\Gamma), \lambda_{2}(\Gamma), \ldots, \lambda_{r}(\Gamma)$ be the distinct eigenvalues of $\Gamma$ with multiplicities $\mu_1, \mu_2, \ldots, \mu_r$, respectively. 
 The \emph{Laplacian spectrum} of $\Gamma$ (or the spectrum  of $\mathcal{L}(\Gamma)$) is represented as
 \begin{center}
$\displaystyle \Phi(\mathcal{L}(\Gamma)) = \begin{pmatrix}
\lambda_{1}(\Gamma) & \lambda_{2}(\Gamma) & \cdots& \lambda_{r}(\Gamma)\\
 \mu_1 & \mu_2 & \cdots & \mu_r
\end{pmatrix}$. 
\end{center}
Sometime we write $\Phi(\mathcal{L}(\Gamma)$  as $\Phi_{L}(\Gamma)$ also.
The Laplacian spectrum of the complete graph $K_m$
 on $m$ vertices and its complement graph $\overline{K}_{m}$
 is given by 
 \begin{align*}
\Phi_{L}(K_{m}) 
 &= \displaystyle \begin{pmatrix}
0 &  m \\
1& m-1\\
\end{pmatrix} 
\end{align*}
and
\begin{align*}
\Phi_{L}({\overline{K}_{m}}) 
 &= \displaystyle \begin{pmatrix}
0 \\
m\\
\end{pmatrix} .
\end{align*}
The following results are useful in the sequel.

\begin{theorem}\cite{MR3009442}\label{laplacianspectrum theoremforjoin}
Let $\Gamma$ be a graph on $k$ vertices having $V(\Gamma) = \{u_1, u_2, \ldots, u_k\}$ and let $\Gamma_1, \Gamma_2, \ldots, \Gamma_k$ be $k$ pairwise disjoint graphs on $n_1, n_2, \ldots, n_k$ vertices, respectively. Then the Laplacian spectrum of $\Gamma[\Gamma_1, \Gamma_2, \ldots, \Gamma_k]$ is given by
\begin{equation}\label{lapacianequation}
\Phi_{L}(\Gamma[\Gamma_1, \Gamma_2, \ldots, \Gamma_k]) = \bigcup\limits_{i=1}^{k} (D_{i} + (\Phi_{L}(\Gamma_i) \setminus \{0\})) \bigcup \Phi(\mathbb{L}(\Gamma))
    \end{equation}
where \[ D_i = \begin{cases} 
      \sum \limits_{u_j \sim u_i}n_j & ~~\text{if}~~ N_{\Gamma}(u_i) \neq \emptyset;\\
      0 & \rm{otherwise} 
    \end{cases}
\]
\begin{equation}
\mathbb{L}({\Gamma})  = \displaystyle \begin{bmatrix}
	D_{1}&  -p_{1,2} & \cdots & -p_{1,k}  \\
	-p_{2,1}& D_{2}   &\cdots&  -p_{2,k} \\ 
 \cdots & \cdots & \cdots & \cdots \\
 -p_{k,1} & -p_{k,2} & \cdots & D_{k}
	\end{bmatrix}
\end{equation}
such that \[ p_{i,j} = \begin{cases} 
      \sqrt{n_in_j} & ~~\text{if}~~ u_i \sim u_j~~ \text{in}~~ \Gamma\\
      0 & \rm{otherwise} 
    \end{cases} \]
in $(\rm\ref{lapacianequation})$, $(\Phi_{L}(\Gamma_i) \setminus \{0\}))$ means that one copy of the eigenvalue $0$ is removed from the multiset $\Phi_{L}(\Gamma_i)$, and $D_i+(\Phi_{L}(\Gamma_i) \setminus \{0\}))$ means $D_i$ is added to each element of $(\Phi_{L}(\Gamma_i) \setminus \{0\}))$.\\

Let $\Gamma$ be a weighted graph by assigning the weight $n_i = |V(\Gamma_i)|$ to the vertex $u_i$ of $\Gamma$ and $i$ varies from $1$ to $k$. Consider $L(\Gamma)= (l_{i,j})$ to be a $k \times k$ matrix, where 
\[ l_{i,j} = \begin{cases}
      -n_j &~~~~\text{if} ~~i \neq j ~~\text{and}~~u_i \sim u_j;\\
      \sum \limits_{u_i \sim u_r}n_r & ~~\text{if}~~ i=j;\\
      0 & \rm{otherwise.} 
    \end{cases}
\]
\end{theorem}

The matrix ${L}(\Gamma)$ is called the vertex-weighted Laplacian matrix of $\Gamma$, which is a zero-row sum matrix but not a symmetric matrix in general.
Though the $k \times k$ matrix $\mathbb{L}(\Gamma)$ defined in Theorem \ref{laplacianspectrum theoremforjoin}, is a symmetric matrix, it need not be a zero row sum matrix. Since the matrices $\mathbb{L}(\Gamma)$  and $L(\Gamma)$ are similar, we have the following remark.
\begin{remark}\label{laplacianremark}
$\Phi(\mathbb{L}(\Gamma)) = \Phi(L(\Gamma))$.
\end{remark}
Let $\mathbb{Z}_n=\{0,1,\ldots, n-1$\} be the ring of integers modulo $n$. The number of integers that are prime to $n$ and less than $n$ is denoted by \emph{Euler totient function} $\phi(n)$. An integer $d$, where $1 < d < n$, is called a proper divisor of $n$ if $d|n$. The number of all the divisors of $n$ is denoted by $\tau(n)$. The greatest common divisor of the two positive integers $a$ and $b$ is denoted by \text{gcd}$(a, b)$. The ideal generated by the element $a$ of $\mathbb{Z}_n$ is the set $\{xa \; : \; x \in \mathbb{Z}_n\}$ and it is denoted by $\langle a \rangle$. 


\section{Structure of the weakly zero-divisor graph $W\Gamma (\mathbb{Z}_n)$}
In this section, we discuss the structure of the weakly zero-divisor graph $W\Gamma(\mathbb{Z}_n)$.  Let $d_1, d_2, \ldots, d_k$ be the proper divisors of $n$. For $1 \leq i \leq k$, consider the following sets
\begin{center}
    $\mathcal{A}_{d_{i}}$= $\{ x \in \mathbb{Z}_{n}: \text{gcd}(x,n) = d_{i}\}$.
\end{center}

\begin{remark}\label{partition}
The sets $\mathcal{A}_{d_{1}}, \mathcal{A}_{d_{2}}, \ldots, \mathcal{A}_{d_{k}}$ form a partition of the vertex set of the graph $W\Gamma(\mathbb{Z}_n)$. Thus, $V(W\Gamma(\mathbb{Z}_n)) = \mathcal{A}_{d_{1}} \cup \mathcal{A}_{d_{2}} \cup \cdots \cup \mathcal{A}_{d_{k}}$. 
\end{remark}

The cardinality of each $\mathcal{A}_{d_i}$ is known in the following lemma.
\begin{lemma}\cite{MR3404655}\label{valueof partition}
 For $1 \leq i \leq k$, we have
$|\mathcal{A}_{d_i}| = \phi\left(\frac{n}{d_i}\right)$. 
\end{lemma}

 \begin{lemma}\label{adjacenyofvertex}
 Let $x \in {\mathcal{A}_{d_i}}$, $ y \in \mathcal{A}_{d_j}$, where $i, j \in \{1, 2, \ldots, k\}$ and $i\neq j$. Then $x \sim y$ in $W\Gamma (\mathbb{Z}_n)$.
 \end{lemma}

\begin{proof}    
Clearly, $x=kd_i$ and $y=k'd_j$ for some positive integers $k, k'$. 
If there exist two distinct prime divisors $p$ and $q$ of $n$  such that $p\mid d_i$ and $q\mid d_j$, then consider $s=\frac{n}{p}$ and $t=\frac{n}{q}$. Note that $s \in {\rm ann}(x)$ and $t\in {\rm ann}(y)$. Also, $st=0$. It follows that $x\sim y$ in $W\Gamma(\mathbb Z_n)$. If there exists a prime divisor $p$ of $n$ such that $d_i=p^{\alpha}$ and $d_j=p^{\beta}$, where  ${ \alpha\neq \beta}$, then choose $s=\frac{n}{p}=t$. Consequently, $s\in {\rm ann}(x)$ and $t\in {\rm ann}(y)$ such that $st=0$. Thus, $x\sim y$ in $W\Gamma(\mathbb Z_n)$.
\end{proof}

 \begin{lemma}\label{annhilatorrelation} 
 For $ x \in \mathcal{A}_{d_i}$, we have $ {\rm ann}(x)$=${\langle \frac{n}{d_i}\rangle }$.
\end{lemma}
\begin{proof}   Clearly, $x=md_i$ for some positive integer $m$. Now let
 $t\in {\rm ann}(x)$. Note that ${\rm ann}(x)={\rm ann}(d_i)$. It follows that $td_i=0$ (mod $n$). Thus, $t=\alpha \frac{n}{d_i}$ for some positive integer $\alpha$. Therefore $t\in\langle\frac{n}{d_i}\rangle$ and so ${\rm ann}(x)\subseteq \langle\frac{n}{d_i}\rangle$. Suppose that $0\neq t\in\langle\frac{n}{d_i}\rangle$. Then $t=k \frac{n}{d_i}$ for some positive integer $k$, where $k\neq 0$ (mod $d_i$). It implies that $tx=0$. Thus, we obtain ${\rm ann}(x)=\langle\frac{n}{d_i}\rangle$.
\end{proof}

\begin{lemma}\label{adjacenyofvertexinjoin} 
For any proper divisor  $d_i$ of $n$, $W\Gamma (\mathcal{A}_{d_i})$ is either a null graph or a complete graph.
\end{lemma}

\begin{proof} 
Suppose $W\Gamma (\mathcal{A}_{d_i})$ is not a complete graph. Then there exist $x, y \in \mathcal{A}_{d_i}$ such that $x \nsim y$. We show that $W\Gamma (\mathcal{A}_{d_i})$ is a null graph. Let $w, z$ be an arbitrary pair of vertices of $W\Gamma (\mathcal{A}_{d_i})$ such that $w \sim z$. Then there exist $s \in {\rm ann}(w)$ and $t\in {\rm ann}(z)$ such that $st=0$. By Lemma \ref{annhilatorrelation}, we have ${\rm ann}(x) = {\rm ann}(y) = {\rm ann}(z) = {\rm ann}(w)$. Consequently, $x \sim y$, which is not possible. Thus, $W\Gamma (\mathcal{A}_{d_i})$ is a null graph.
\end{proof}

If $n={p_1}^{k_1}{p_2}^{k_2}\ldots{p_m}^{k_m},$ where $k_i\geq 2$, then by  \rm \cite[Theorem 2.6]{nikmehr2021weakly}, $W\Gamma(\mathbb{Z}_n)$ is a complete graph. Also, if $n = p$, then clearly $W\Gamma(\mathbb{Z}_p)$ is an empty graph. Therefore, in the remaining paper, we compute the Laplacian spectrum of $W\Gamma(\mathbb{Z}_n)$ for $n = p_1p_2\ldots p_m{q_1}^{k_1}{q_2}^{k_2}\ldots{q_r}^{k_r} $, where $k_i\geq 2,m\geq1$ and $r\geq0$. 
 
\begin{lemma}\label{adjacenyofjoin}
Let $D$= $\{ d_1,d_2,\ldots, d_{k}\}$ be the set of all proper divisors of $n$ and $n = p_1p_2\ldots p_m{q_1}^{k_1}{q_2}^{k_2}\ldots{q_r}^{k_r} $, where $k_i\geq 2,m\geq 1$ and $r\geq 0 $. Then the subgraph of $W\Gamma ({\mathbb Z_n}) $  induced by $\mathcal{A}_{d_i}$ is $\overline{K}_{\phi\left({\frac{n}{d_i}}\right)}$ if and only if $d_i\in \{p_1,p_2,\ldots ,p_m\}$.
\end{lemma}

\begin{proof}
 First suppose that $W\Gamma (\mathcal{A}_{ d_i})=\overline{K}_{\phi\left({\frac{n}{d_i}}\right)}$. On contrary, assume that $d_i\notin \{p_1,p_2,\ldots,p_m\}$. Let $x,y\in \mathcal{A}_{d_i}$. If there exist two prime divisors $p,q$ of $d_i$, then on taking $s=\frac{n}{p}$ and $t=\frac{n}{q}$, we get $s\in {\rm ann}(x)$ and $t\in {\rm ann}(y)$ such that $st=0$. It follows that $x\sim y$, which is not possible.
 If there exist a prime divisor $p$ of $d_i$, then consider $s=\frac{n}{p}=t$. One can observe that $s\in {\rm ann}(x)$ and $t\in {\rm ann}(y)$ such that $st=0$. It implies that $x\sim y$, again a contradiction. Therefore, $d_i\in \{p_1,p_2,\ldots,p_m\}$.
Conversely, suppose that $d_i\in \{p_1,p_2,\ldots,p_m\}$ i.e. $d_i=p_k$ for some $k$. Let $x,y\in \mathcal{A}_{d_i} $ and let $0\neq s\in {\rm ann}(x)$ and $0\neq t\in {\rm ann}(y)$. Then $s=\alpha \frac{n}{p_{k}}$ and $t=\beta\frac{n}{p_k}$, where $\alpha\neq0$ (mod $p_k$) and $\beta\neq0$ (mod $p_k$). Note that $st\neq 0$ (mod $n$) and so $x\nsim y$ in $W\Gamma(\mathcal{A}_{d_i})$. Therefore, $W\Gamma (\mathcal{A}_ {d_i})$ is not a complete graph. By Lemma \ref{adjacenyofvertexinjoin}, we obtain $W\Gamma (\mathcal{A}_ {d_i})=\overline{K}_{\phi\left({\frac{n}{d_i}}\right)}$.
\end{proof}

\begin{corollary}\label{partitionofcozerodivisorgraphisomorphic}
The following statements hold:
\begin{itemize}
    \item[(i)] For $i \in \{1, 2, \ldots, k\}$, the  subgraph of $W\Gamma(\mathbb{Z}_n)$ induced by $\mathcal{A}_{d_i}$ is isomorphic to either $\overline{K}_{\phi\left(\frac{n}{d_i}\right)}$ or ${K}_{\phi\left(\frac{n}{d_i}\right)}.$
    \item [(ii)] For $i,j \in \{1, 2, \ldots, k\}$ and $i \neq j$, a vertex of $\mathcal{A}_{d_i}$ is adjacent to  all the vertices of $\mathcal{A}_{d_j}$. 
\end{itemize}
\end{corollary}

We define $\Upsilon_n$ by a complete graph on the set $\{d_1,d_2,\ldots,d_k\}$ of all proper divisors of $n$.

\begin{lemma}\label{inducedsubgraphequaltogamma} 
Let  $W\Gamma(\mathcal{A}_{d_i})$ be the subgraph of $W\Gamma(\mathbb{Z}_{n})$ induced by $\mathcal{A}_{d_i}$ $(1\leq i\leq k)$. Then \[W\Gamma(\mathbb{Z}_n) = \Upsilon_n [W\Gamma(\mathcal{A}_{d_1}), W\Gamma (\mathcal{A}_{d_2}), \ldots, W\Gamma(\mathcal{A}_{d_k})].\]
\end{lemma}

\begin{proof}
The result can be obtained by replacing the vertex $d_i$ of $\Upsilon_n$ by $W\Gamma(\mathcal{A}_{d_i})$, for $1 \leq i \leq k$ and by using Lemma \ref{adjacenyofvertex}. 
\end{proof}

\begin{example}
The weakly zero-divisor graph $W\Gamma(\mathbb{Z}_{18})$ is shown in Figure $1$. 

By Lemma \ref{inducedsubgraphequaltogamma}, note that  $W\Gamma(\mathbb{Z}_{18}) = \Upsilon_{18} [W\Gamma(\mathcal{A}_{2})$, 
$W\Gamma(\mathcal{A}_{3})$, $W\Gamma(\mathcal{A}_{6})$, 
$W\Gamma(\mathcal{A}_{9})$], where $\Upsilon_{18}$ is complete graph on the set $\{2,3,6,9\}$ and  $W\Gamma(\mathcal{A}_{2})$ = $\overline{K}_6$, $W\Gamma(\mathcal{A}_{3})$ = ${K}_2$ = $W\Gamma(\mathcal{A}_{6})$, $W\Gamma(\mathcal{A}_{9})$ = ${K}_1 $. 
\begin{figure}[h!]
\centering
\includegraphics[width=0.4 \textwidth]{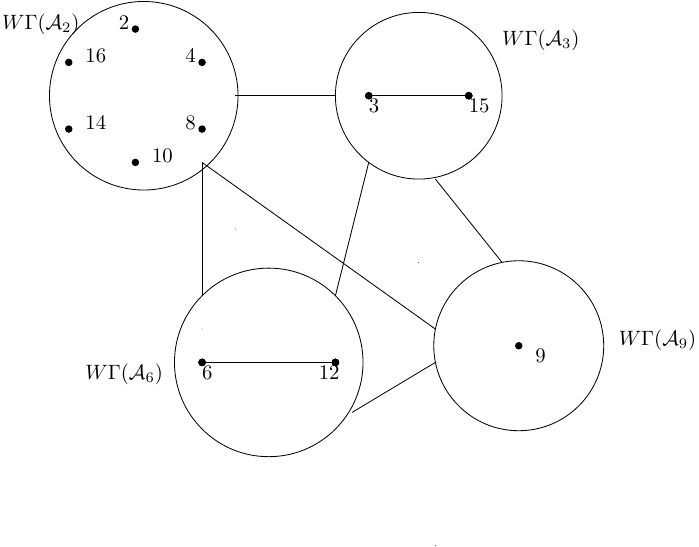}
\caption{The graph $W\Gamma(\mathbb{Z}_{18})$}
\end{figure}
\end{example}

\section{Laplacian spectrum of $W\Gamma(\mathbb{Z}_n)$}

In this section, we compute the Laplacian spectrum of the graph $W\Gamma(\mathbb{Z}_n)$ for arbitrary $n$.  For each $1 \leq i \leq k$, we assign a weight to the vertex $d_i$ of the graph $\Upsilon_n$ and this weight is determined by the value of the Euler's totient function applied to $\frac{n}{d_i}$. We denote this weight as $\phi\left(\frac{n}{d_i}\right)$ which is equivalent to the cardinality of the set $\mathcal{A}_{d_i}$.

The $k \times k$ weighted Laplacian matrix $L(\Upsilon_{n})$ of $\Upsilon_{n}$, defined in Theorem \ref{laplacianspectrum theoremforjoin}, is given by 
\begin{equation}
L(\Upsilon_{n})  = \displaystyle \begin{bmatrix}\label{spectrum matrix2}
	D_{d_1}&  -l_{1,2} & \cdots & -l_{1,k}  \\
	-l_{2,1}& D_{d_2}   &\cdots&  -l_{2,k} \\ 
 \cdots & \cdots & \cdots & \cdots \\
 -l_{k,1} & -l_{k,2} & \cdots & D_{d_k}
	\end{bmatrix},
\end{equation}
where \[ l_{i,j} = \begin{cases} 
      \phi\left(\frac{n}{d_j}\right) & ~~\text{if}~~ d_i \sim d_j~~ \text{in}~~ \Upsilon_{n};\\
      0 & \rm{otherwise.} 
    \end{cases}
\]
and \begin{center}
    $D_{d_j} =  \sum \limits_{{d_i} \in N_{\Upsilon_n}(d_j)} \phi\left(\frac{n}{d_i}\right)$.
\end{center}
\begin{theorem}\label{laplacianeigenvaluescozero}
The Laplacian spectrum of $W\Gamma(\mathbb{Z}_n)$ is given by \begin{center}
    $\Phi_{L}(W\Gamma(\mathbb{Z}_n)) = \bigcup\limits_{i=1}^{k} (D_{d_i} + (\Phi_{L}(W\Gamma(\mathcal{A}_{d_i})) \setminus \{0\})) \bigcup \Phi(L(\Upsilon_{n}))$,
\end{center}
where $D_{d_i} + (\Phi_{L}(W\Gamma(\mathcal{A}_{d_i})) \setminus \{0\})$ represents that $D_{d_i}$ is added to each element of the multi-set $(\Phi_{L}(W\Gamma(\mathcal{A}_{d_i})) \setminus \{0\})$.
\end{theorem}

\begin{proof}
By Lemma \ref{inducedsubgraphequaltogamma}, $W\Gamma(\mathbb{Z}_n) = \Upsilon_n [W\Gamma(\mathcal{A}_{d_1}), W\Gamma(\mathcal{A}_{d_2}), \ldots, W\Gamma(\mathcal{A}_{d_k})]$. Consequently, by Theorem \ref{laplacianspectrum theoremforjoin} and Remark \ref{laplacianremark}, the result holds.
\end{proof}



\begin{example}
We discuss the Laplacian spectrum of  $W\Gamma(\mathbb{Z}_n)$ for ${\rm (i)}$ $n=pqr$ and ${\rm (ii)}$ $n = p^kq $ $(k\geq 2)$, where $p$ and $q$ are distinct primes.

\begin{itemize}
    \item[(i)] Let $n=pqr$. First observe that $\Upsilon_{pqr}$ is the complete graph on $6$ vertices. Consequently, by Lemma \ref{inducedsubgraphequaltogamma}, we have  $\Gamma(\mathbb{Z}_{pqr}) = \Upsilon_{pqr}[W\Gamma(\mathcal{A}_{p}), W\Gamma(\mathcal{A}_{q}), W\Gamma(\mathcal{A}_{r}), W\Gamma(\mathcal{A}_{pq}),  W\Gamma(\mathcal{A}_{qr}), W\Gamma(\mathcal{A}_{pr}) ],$ where $W\Gamma(\mathcal{A}_{p}) =\overline {K}_{\phi(qr)}$,  $W\Gamma(\mathcal{A}_{q}) = \overline{K}_{\phi(pr)}$, $W\Gamma(\mathcal{A}_{r}) =\overline {K}_{\phi(pq)}$, $W\Gamma(\mathcal{A}_{pq}) = {K}_{\phi(r)}$, 
$W\Gamma(\mathcal{A}_{qr}) = {K}_{\phi(p)}$ and
$W\Gamma(\mathcal{A}_{pr}) = {K}_{\phi(q)}$. The cardinality $|V|$ of the vertex set $V$ of $W\Gamma(\mathbb{Z}_{pqr})$ is given by ${\phi(p)+\phi(q)+\phi(r)+\phi(pq)+\phi(qr)+\phi(pr)}$. It follows that $D_p = \phi(pr)+\phi(pq)+\phi(p)+\phi(q)+\phi(r)= |V|-\phi(qr)$ and $D_q = |V|-\phi(pr)$, $D_{r} =| V|-\phi(pq)$, $D_{pq} = |V|-\phi(r)$, $D_{qr} =| V|-\phi(p)$ and  $D_{pr} =| V|-\phi(q)$. Therefore, by Theorem \ref{laplacianeigenvaluescozero}, the Laplacian spectrum of $W\Gamma(\mathbb{Z}_{pqr})$ is 
\begin{align*}
\Phi_{L}(W\Gamma(\mathbb{Z}_{pqr})) &= (D_{p} + (\Phi_{L}(W\Gamma(\mathcal{A}_{p})) \setminus \{0\})) \bigcup (D_{q} + (\Phi_{L}(W\Gamma(\mathcal{A}_{q})) \setminus \{0\})) \bigcup (D_{r} + (\Phi_{L}(W\Gamma(\mathcal{A}_{r})) \setminus \{0\}))\\
& \bigcup (D_{pq} + (\Phi_{L}(W\Gamma(\mathcal{A}_{pq})) \setminus \{0\})) \bigcup (D_{qr} + (\Phi_{L}(W\Gamma(\mathcal{A}_{qr})) \setminus \{0\})) \bigcup (D_{pr} + (\Phi_{L}(W\Gamma(\mathcal{A}_{pr})) \setminus \{0\}))\\ & \bigcup \Phi(L(\Upsilon_{pqr}))\\
 &= \displaystyle \begin{pmatrix}
|V| &  |V|-\phi(pq) & |V|-\phi(qr)& |V|-\phi(pr) \\
\phi(p)+\phi(q)+\phi(r)-3& \phi(pq)-1&\phi(qr)-1&\phi(pr)-1 \\
\end{pmatrix} \bigcup \Phi(L(\Upsilon_{pqr})).
\end{align*}

Thus, the remaining $6$ Laplacian eigenvalues can be obtained by the characteristic polynomial of the matrix 
\[L(\Upsilon_{pqr})  =\displaystyle \begin{bmatrix}
	|V|-\phi(qr) & -\phi(pr) & -\phi(pq) & -\phi(r) & -\phi(p) &-\phi(q)\\
	-\phi(qr) &|V|-\phi(pr)  & -\phi(pq) & -\phi(r)  & -\phi(p) &-\phi(q)\\
	-\phi(qr) & -\phi(pr) & |V|-\phi(pq)  & -\phi(r)  & -\phi(p) &-\phi(q)\\
	-\phi(qr) & -\phi(pr) & -\phi(pq) & |V|-\phi(r) & -\phi(p) &-\phi(q) \\
 -\phi(qr) & -\phi(pr) & -\phi(pq) & -\phi(r) & |V|-\phi(p) &-\phi(q) \\
 -\phi(qr) & -\phi(pr) & -\phi(pq) & -\phi(r) & -\phi(p) & |V|-\phi(q) 
\end{bmatrix},\]
where the matrix $L(\Upsilon_{pqr})$ is obtained by indexing the rows and columns as $p, q,r, pq, qr,pr.$
\begin{equation*}
L(\Upsilon_{pqr})=
\begin{bmatrix}
	|V| & 0 & 0 & 0 & 0 & 0\\
         0 & |V| & 0 & 0& 0 & 0\\
         0 &  0 & |V|  & 0 & 0 & 0\\
	0 & 0 & 0 & |V|& 0 & 0\\
        0 & 0 & 0 & 0& |V| & 0\\
        0 & 0 & 0 & 0& 0 & |V|\\
\end{bmatrix}
+
\begin{bmatrix}
	-\phi(qr) & -\phi(pr) & -\phi(pq) & -\phi(r) & -\phi(p) &-\phi(q) \\
	-\phi(qr) &-\phi(pr)  & -\phi(pq) & -\phi(r)  & -\phi(p) &-\phi(q)\\
	-\phi(qr) & -\phi(pr) & -\phi(pq)  & -\phi(r)  & -\phi(p) &-\phi(q)\\ 
	-\phi(qr) & -\phi(pr) & -\phi(pq) & -\phi(r) & -\phi(p) &-\phi(q) \\
 
     -\phi(qr) & -\phi(pr) & -\phi(pq) & -\phi(r) & -\phi(p) &-\phi(q) \\
       -\phi(qr) & -\phi(pr) & -\phi(pq) & -\phi(r) & -\phi(p) & -\phi(q) 
\end{bmatrix}.
\end{equation*} \\
Note that the eigenvalue of the first matrix is $|V|$ with multiplicity $6$ and the second matrix is of rank one. Consequently, the eigenvalues of the second matrix are $-|V|$ with multiplicity $1$ and $0$ with multiplicity $5$, respectively. Therefore, the eigenvalues of $L(\Upsilon_{pqr})$ are $|V|$ with multiplicity $5$ and $0$ with multiplicity $1$, respectively. Hence, the Laplacian spectrum of $W\Gamma(\mathbb{Z}_{pqr})$ is
\[\displaystyle \begin{pmatrix}
0 & |V| &  |V|-\phi\left( qr\right  ) &|V|-\phi\left( pr\right  )&|V|-\phi\left( pq\right  )\\
 1 & \phi\left( p\right)+\phi\left(q\right)+\phi\left(r\right)+2 & \phi\left(qr\right)-1&\phi\left(pr\right)-1&\phi\left(pq\right)-1
\end{pmatrix}.\]
    
\item[(ii)] Let $n = p^kq \; (k \geq 2)$. Note that $\{p, p^2, \ldots, p^{k}, q, pq, p^2q, \ldots, p^{k-1}q\}$ is the vertex set of the graph $\Upsilon_{p^{k}q}$. By Lemma \ref{inducedsubgraphequaltogamma},
\[W\Gamma(\mathbb{Z}_{p^{k}q}) = \Upsilon_{p^{k}q}[W\Gamma(\mathcal{A}_{p}), W\Gamma(\mathcal{A}_{p^2}), \ldots, W\Gamma(\mathcal{A}_{p^{k}}), W\Gamma(\mathcal{A}_{q}), W\Gamma(\mathcal{A}_{pq}), W\Gamma(\mathcal{A}_{p^2q}), \ldots, W\Gamma(\mathcal{A}_{p^{k-1}q})],\]
where $W\Gamma(\mathcal{A}_{p}) = {K}_{\phi(p^{k-1}q)}$, $W\Gamma(\mathcal{A}_{p^2}) = {K}_{\phi(p^{k-2}q)}$, $\ldots$, $W\Gamma(\mathcal{A}_{p^{k}}) = {K}_{\phi(q)}$, $W\Gamma(\mathcal{A}_{q}) = \overline{K}_{\phi(p^{k})}$, $W\Gamma(\mathcal{A}_{pq}) = {K}_{\phi(p^{k-1})}$, $\ldots$, $W\Gamma(\mathcal{A}_{p^{k-1}q}) = {K}_{\phi(p)}$.
 The cardinality $|V|$ of the vertex set $V$ of $W\Gamma(\mathbb{Z}_{p^kq})$ is given by \[ \phi(p^{k-1}q) + \phi(p^{k-2}q) + \phi(p^{k-3}q) + \cdots + \phi(q) + \phi(p^{k-1}) + \phi(p^{k-2}) + \phi(p^{k-3}) + \cdots + \phi(p) + \phi(p^k).\]
It follows that
$D_p = \phi(p^{k})+\phi(p^{k-2}q)+\phi(p^{k-3}q)+\cdots+\phi(q)+\phi(p^{k-1})+\phi(p^{k-2})+\cdots +\phi(p)=|V|-\phi(p^{k-1}q)$, $D_{p^2}=|V|-\phi(p^{k-2}q)$,$\ldots$, $D_{p^{k}} =| V|-\phi(q)$,  $D_{pq} = | V|-\phi(p^{k-1})$,$\ldots,$ $D_{p^rq} = | V|-\phi(p^{k-r})$,\ldots, $D_{p^{k-1}q} =  |V|-\phi(p)$, $D_q=|V|-\phi(p^{k})$ 
Consequently, by Theorem \ref{laplacianeigenvaluescozero}, the Laplacian spectrum $\Phi_{L}(W\Gamma(\mathbb{Z}_{p^{k}q}))$ is
\begin{align*}
 &= (D_{p} + (\Phi_{L}(W\Gamma(\mathcal{A}_{p})) \setminus \{0\})) \bigcup (D_{p^2} + (\Phi_{L}(W\Gamma(\mathcal{A}_{p^2})) \setminus \{0\})) \bigcup \cdots \bigcup (D_{p^{k}} + (\Phi_{L}(W\Gamma(\mathcal{A}_{p^{k}})) \setminus \{0\})) \\
&  \bigcup (D_{q} + (\Phi_{L}(W\Gamma(\mathcal{A}_{q})) \setminus \{0\})) \bigcup (D_{pq} + (\Phi_{L}(W\Gamma(\mathcal{A}_{pq})) \setminus \{0\})) \bigcup (D_{p^2q} + (\Phi_{L}(W\Gamma(\mathcal{A}_{p^2q})) \setminus \{0\}))\\
&\bigcup \cdots \bigcup (D_{p^{k-1}q} + (\Phi_{L}(W\Gamma(\mathcal{A}_{p^{k-1}q})) \setminus \{0\})) 
  \bigcup \Phi(L(\Upsilon_{p^{k}q})) \\
&= \displaystyle \begin{pmatrix}
|V| &  |V|-\phi(p^k)\\
\left(\sum \limits_{i=1}^{k}\phi(p^{k-i}q)+\sum \limits_{i=1}^{k-1}\phi(p^{k-i})\right)-(2k-1) &\phi(p^k)-1  
\end{pmatrix} \bigcup \Phi(L(\Upsilon_{p^{k}q})).
\end{align*}

Thus, the remaining $2k$ Laplacian eigenvalues can be obtained by characteristic polynomial of the following matrix $L(\Upsilon_{p^{k}q})$ \\
\[ = \begin{bmatrix}
|V|-\phi(p^{k-1}q) &  -\phi(p^{k-2}q) &  \cdots &- \phi(q )& -\phi(p^{k})&-\phi(p^{k-1}) & \cdots &-\phi(p^2) &-\phi(p)\\
-\phi(p^{k-1}q) &  |V|-\phi(p^{k-2}q) &   \cdots & -\phi(q ) & -\phi(p^{k}) & -\phi(p^{k-1}) & \cdots &-\phi(p^2) &-\phi(p) \\
\vdots & \vdots &  \vdots & \vdots & \vdots & \vdots & \vdots & \vdots & \vdots \\
-\phi(p^{k-1}q) &   -\phi(p^{k-2}q) &   \cdots & |V|-\phi(q) & -\phi(p^{k}) & -\phi(p^{k-1}) & \cdots  &-\phi(p^2) &-\phi(p)\\
-\phi(p^{k-1}q) &   -\phi(p^{k-2}q) &   \cdots & -\phi(q) & |V|-\phi(p^{k}) & -\phi(p^{k-1}) & \cdots  &-\phi(p^2) &-\phi(p)\\

 -\phi(p^{k-1}q) &  -\phi(p^{k-2}q) &   \cdots &- \phi(q )&-\phi(p^{k})&|V|-\phi(p^{k-1}) & \cdots& -\phi(p^2) &-\phi(p)\\
\vdots & \vdots &  \vdots & \vdots & \vdots & \vdots & \vdots & \vdots & \vdots \\
-\phi(p^{k-1}q) &  -\phi(p^{k-2}q) &  \cdots &- \phi(q )& -\phi(p^{k})&\cdots &|V|-\phi(p^{k-r})&\cdots &-\phi(p)\\
\vdots & \vdots &  \vdots & \vdots & \vdots & \vdots & \vdots & \vdots & \vdots \\
-\phi(p^{k-1}q) &  -\phi(p^{k-2}q) &  \cdots &- \phi(q )& -\phi(p^{k})&-\phi(p^{k-1}) & \cdots &-\phi(p^2) &|V|-\phi(p)
\end{bmatrix},\] 
\vspace{.1cm}
where the matrix $L(\Upsilon_{p^{k}q})$ is obtained by indexing the rows and columns as $p, p^2, \ldots, p^{k},q, pq,p^2q,\ldots,p^{k-2}q,p^{k-1}q$. The matrix $ L(\Upsilon_{p^{k}q})$ can be decomposed into $A + B$, where $A = |V|I$  with the identity matrix $I$ of order $2k$ and 
\[B = \begin{bmatrix}
-\phi(p^{k-1}q) &  -\phi(p^{k-2}q) &  \cdots &- \phi(q )& -\phi(p^{k})&-\phi(p^{k-1}) & \cdots &-\phi(p^2) &-\phi(p)\\
-\phi(p^{k-1}q) &  -\phi(p^{k-2}q) &   \cdots & -\phi(q ) & -\phi(p^{k}) & -\phi(p^{k-1}) & \cdots &-\phi(p^2) &-\phi(p) \\
\vdots & \vdots &  \vdots & \vdots & \vdots & \vdots & \vdots & \vdots & \vdots \\
-\phi(p^{k-1}q) &   -\phi(p^{k-2}q) &   \cdots & -\phi(q) & -\phi(p^{k}) & -\phi(p^{k-1}) & \cdots  &-\phi(p^2) &-\phi(p)\\
-\phi(p^{k-1}q) &   -\phi(p^{k-2}q) &   \cdots & -\phi(q) & -\phi(p^{k}) & -\phi(p^{k-1}) & \cdots  &-\phi(p^2) &-\phi(p)\\

 -\phi(p^{k-1}q) &  -\phi(p^{k-2}q) &   \cdots &- \phi(q )&-\phi(p^{k})&-\phi(p^{k-1}) & \cdots& -\phi(p^2) &-\phi(p)\\
\vdots & \vdots &  \vdots & \vdots & \vdots & \vdots & \vdots & \vdots & \vdots \\
-\phi(p^{k-1}q) &  -\phi(p^{k-2}q) &  \cdots &- \phi(q )& -\phi(p^{k})&\cdots &-\phi(p^{k-r})&\cdots &-\phi(p)\\
\vdots & \vdots &  \vdots & \vdots & \vdots & \vdots & \vdots & \vdots & \vdots \\
-\phi(p^{k-1}q) &  -\phi(p^{k-2}q) &  \cdots &- \phi(q )& -\phi(p^{k})&-\phi(p^{k-1}) & \cdots &-\phi(p^2) &-\phi(p)
\end{bmatrix}.\]

In the matrix $A$, the eigenvalue $|V|$ appears with multiplicity $2k$ and the matrix $B$ is of rank one. Consequently, the eigenvalues of the matrix $B$ are $-|V|$ and $0$, with  multiplicities  $1$ and $2k-1$, respectively. Therefore, the eigenvalues of $L(\Upsilon_{p^{k}q})$ are $|V|$ and $0$ with multiplicity $2k-1$ and $1$, respectively.

Hence, the Laplacian spectrum of $W\Gamma(\mathbb{Z}_{p^{k}q})$ is
\[\displaystyle \begin{pmatrix}
0 & |V| &  |V|-\phi(p^k)  \\
 1 & \left(\sum \limits_{i=1}^{k}\phi(p^{k-i}q)+\sum \limits_{i=1}^{k-1}\phi(p^{k-i})\right) & \phi(p^k)-1 \\
\end{pmatrix}.\]
\end{itemize}
\end{example}

\begin{theorem}
Let $n = p_1p_2\cdots p_m{q_1}^{k_1}{q_2}^{k_2}\cdots{q_r}^{k_r} $ $(k_i\geq 2,m\geq 1,r\geq0)$, where ${p_{i}}'s$ and ${q_{i}}'s$ are distinct primes. Suppose $D=\{d_1,d_2,\ldots,d_{\tau(n)-2}\}$ is the set of all proper divisors of $n$. Then the  Laplacian spectrum  of $W\Gamma(\mathbb{Z}_n)$ is given by
\[\displaystyle \begin{pmatrix}
0 & |V| & | V|-\phi\left(\frac{n}{p_1}\right)& |V|-\phi\left(\frac{n}{p_2}\right)&\cdots&|V|-\phi\left(\frac{n}{p_m}\right) \\
 1 &\left(\sum \limits_{d_i\neq p_i}\phi\left( \frac{n}{d_i}\right)+m-1\right) &\phi\left(\frac{n}{p_1}\right)-1&\phi\left(\frac{n}{p_2}\right)-1&\cdots&\phi\left(\frac{n}{p_m}\right)-1 \ \\
\end{pmatrix},\]\\
where $V$ is the  vertex set of $W\Gamma(\mathbb{Z}_n)$. Indeed, the weakly zero-divisor graph $W\Gamma(\mathbb{Z}_n)$ of the ring $\mathbb{Z}_{n}$ is Laplacian integral.
\end{theorem}

\begin{proof} Let $A'=\{p_1,p_2,\ldots, p_m\}$.
 By Theorem \ref{laplacianeigenvaluescozero}, the  Laplacian spectrum of  $W\Gamma(\mathbb{Z}_{n})$ is
\begin{align*}
\Phi_{L}(W\Gamma(\mathbb{Z}_{n})) &= \bigcup_{d_i\in A} (D_{d_i} + (\Phi_{L}(W\Gamma(\mathcal{A}_{d_i})) \setminus \{0\})) \bigcup_{d_i\notin A} (D_{d_{j}} + (\Phi_{L}(W\Gamma(\mathcal{A}_{d_{j}})) \setminus \{0\}))  \bigcup \Phi(L(\Upsilon_{n})). 
\end{align*}
By Lemma \ref{valueof partition}, \ref{adjacenyofjoin} and Corollary \ref{partitionofcozerodivisorgraphisomorphic}, for each $d_i\in A'$, we have $W\Gamma(\mathcal{A}_{d_{i}}) = \overline{K}_{\phi \left(\frac{n}{d_i}\right)}$ and $W\Gamma(\mathcal{A}_{d_j}) =  {K}_{\phi\left(\frac{n}{d_j}\right)}$, where $d_j\notin A'$.
One can observe that the cardinality $|V|$ of the vertex set $V$ of $W\Gamma(\mathbb{Z}_{n})$ is $\left(\sum \limits_{i=1}^{\tau(n)-2}\phi\left(\frac{n}{d_{i}}\right)\right)$.
Also, note that for $1\leq i\leq{\tau(n)-2}$, we have

\[D_{d_i} 
 =\sum \limits_{\substack{j=1\\j\neq i}}^{\tau(n)-2}\phi\left(\frac{n}{d_{j}}\right)
 =|V|-\phi\left (\frac{n}{d_i}\right).\]
Thus, we obtain $\Phi_{L}(W\Gamma(\mathbb{Z}_{n}))$
\begin{eqnarray*}
  &= \displaystyle \bigcup_{d_i\in A'}\left( \left( |V|-\phi\left(\frac{n}{d_i}\right)\right) + \left(\Phi_{L}\left(\overline{K}_{\phi\left(\frac{n}{d_i}\right)}\right) \setminus \{0\}\right)\right) \displaystyle \bigcup_{d_j\notin A'}\left( \left(|V|-\phi\left(\frac{n}{d_j}\right)\right) + \left(\Phi_{L}\left(K_{\phi\left(\frac{n}{d_j}\right)}\right) \setminus \{0\}\right)\right)  \bigcup \Phi(L(\Upsilon_{n})) \\ 
  &= \displaystyle \bigcup_{d_i\in A'}\left( \left( |V|-\phi\left(\frac{n}{d_i}\right)\right) +\left(\begin{pmatrix} 0 \\
\phi\left(\frac{n}{d_i}\right)\\
\end{pmatrix}  \setminus \{0\}\right)\right) \displaystyle \bigcup_{d_j\notin A'}\left( \left(|V|-\phi\left(\frac{n}{d_j}\right)\right) +\left( \begin{pmatrix}
0 &  \phi\left(\frac{n}{d_j}\right) \\
1& \phi\left(\frac{n}{d_j}\right)-1\\
\end{pmatrix}  \setminus \{0\}\right)\right)  \displaystyle \bigcup \Phi(L(\Upsilon_{n}))\\
&= \displaystyle \bigcup_{d_i\in A'}\left( \left( |V|-\phi\left(\frac{n}{d_i}\right)\right) +\begin{pmatrix} 0 \\
\phi\left(\frac{n}{d_i}\right)-1\\
\end{pmatrix}  \right) \bigcup_{d_j\notin A'}\left( \left(|V|-\phi\left(\frac{n}{d_j}\right)\right) + \begin{pmatrix}
  \phi\left(\frac{n}{d_j}\right) \\
 \phi\left(\frac{n}{d_j}\right)-1\\
\end{pmatrix}  \right) \displaystyle \bigcup \Phi(L(\Upsilon_{n})) \\
&= \displaystyle \bigcup_{d_i\in A'}\begin{pmatrix}
|V|- \phi\left(\frac{n}{d_i}\right)\\
\phi\left(\frac{n}{d_i}\right)-1\\
\end{pmatrix} \displaystyle \bigcup_{d_j\notin A'} \begin{pmatrix}
  |V| \\
 \phi\left(\frac{n}{d_j}\right)-1\\
\end{pmatrix}    \bigcup \Phi(L(\Upsilon_{n})) \\
& =\begin{pmatrix}
|V| & | V|-\phi\left(\frac{n}{p_1}\right)&|V|-\phi\left(\frac{n}{p_2}\right)&\cdots&  |V|-\phi\left(\frac{n}{p_m}\right) \\
\sum \limits_{d_i\neq p_i}\phi\left(\frac{n}{d_i}\right)-(\tau(n)-2-m) &\phi\left(\frac{n}{p_1}\right)-1& \phi\left(\frac{n}{p_2}\right)-1&\cdots& \phi\left(\frac{n}{p_m}\right)-1
\end{pmatrix} \bigcup \Phi(L(\Upsilon_{n})).
\end{eqnarray*}
Thus, the remaining  Laplacian eigenvalues can be obtained by characteristic polynomial of the following matrix\\
\[L(\Upsilon_{n}) =  
 \displaystyle \begin{bmatrix}
|V|-\phi\left(\frac{n}{d_1}\right) &  -\phi\left(\frac{n}{d_2}\right) &\cdots & \cdots &- \phi\left(\frac{n}{d_m}\right)& -\phi\left(\frac{n}{d_{m+1}}\right)&\cdots&\cdots & -\phi\left(\frac{n}{d_{\tau(n)-2}}\right)\\
-\phi\left(\frac{n}{d_1}\right) &  |V|-\phi\left(\frac{n}{d_2}\right) &\cdots & \cdots &-\phi\left(\frac{n}{d_m}\right) & -\phi\left(\frac{n}{d_{m+1}}\right) & \cdots &\cdots &-\phi\left(\frac{n}{d_{\tau(n)-2}}\right) \\
\vdots & \vdots &  \vdots & \vdots & \vdots & \vdots & \vdots & \vdots & \vdots \\
\vdots & \vdots &  \vdots & \vdots & \vdots & \vdots & \vdots & \vdots & \vdots \\
\vdots & \vdots &  \vdots & \vdots & \vdots & \vdots & \vdots & \vdots & \vdots \\
-\phi\left(\frac{n}{d_1}\right) &  -\phi\left(\frac{n}{d_2}\right) &\cdots  &\cdots & |V|-\phi\left(\frac{n}{d_m}\right)& -\phi\left(\frac{n}{d_{m+1}}\right)&\cdots&\cdots & -\phi\left(\frac{n}{d_{\tau(n)-2}}\right)\\
-\phi\left(\frac{n}{d_1}\right) &  -\phi\left(\frac{n}{d_2}\right) &  \cdots&\cdots&- \phi\left(\frac{n}{d_m}\right)& |V|-\phi\left(\frac{n}{d_{m+1}}\right)&\cdots & \cdots & -\phi\left(\frac{n}{d_{\tau(n)-2}}\right)\\

\vdots & \vdots &  \vdots & \vdots & \vdots & \vdots & \vdots & \vdots & \vdots \\
\vdots & \vdots &  \vdots & \vdots & \vdots & \vdots & \vdots & \vdots & \vdots \\
\vdots & \vdots &  \vdots & \vdots & \vdots & \vdots & \vdots & \vdots & \vdots \\
-\phi\left(\frac{n}{d_1}\right) &  -\phi\left(\frac{n}{d_2}\right) &  \cdots&\cdots&- \phi\left(\frac{n}{d_m}\right)& -\phi\left(\frac{n}{d_{m+1}}\right)&\cdots & \cdots & |V|-\phi\left(\frac{n}{d_{\tau(n)-2}}\right)\\
\end{bmatrix},\]\\
where matrix $L(\Upsilon_{{n}})$ is obtained by indexing the rows and columns as $d_1, d_2, \ldots,d_m,  d_{m+1}, \ldots, d_{\tau(n)-2}$, where $d_i= p_i$ $(1\leq i\leq m)$. Observe that $L(\Upsilon_{{n}})=|V|I+B$, where $I$ is the identity matrix of order $\tau(n)-2$ and the matrix $B$ is of rank one such that the sum of each row of $B$ is equal to $-|V|$. Thus, the Laplacian spectrum of $L(\Upsilon_{n})$ is\\ 
\[\displaystyle \begin{pmatrix}
0 & |V|  \\
 1 &\tau(n)-3  \\
\end{pmatrix}.\]
 Therefore, the Laplacian spectrum of $W\Gamma(\mathbb{Z}_n)$ is
\[\displaystyle \begin{pmatrix}
0 & |V| & | V|-\phi\left(\frac{n}{p_1}\right)& |V|-\phi\left(\frac{n}{p_2}\right)&\cdots&|V|-\phi\left(\frac{n}{p_m}\right) \\
 1 &\left(\sum \limits_{d_i\neq p_i}\phi\left(\frac{n}{d_i}\right)+m-1\right) &\phi\left(\frac{n}{p_1}\right)-1&\phi\left(\frac{n}{p_2}\right)-1&\cdots&\phi\left(\frac{n}{p_m}\right)-1 \ \\
\end{pmatrix}.\]
\end{proof}

\textbf{Acknowledgement:} The first and second author gratefully acknowledge Birla Institute of Technology and Science (BITS) Pilani, India, for providing financial support.

\end{document}